\documentclass[11pt]{article}

\usepackage{amscd, amsmath, amssymb, amsfonts, amsthm}
\usepackage{euscript}
\usepackage{graphicx}
\usepackage{pdfpages}
\usepackage[square, comma, sort&compress, numbers]{natbib}

\let\oldref=\ref
\def\ref#1{(\oldref{#1})}

\def\theenumi{\roman{enumi}}

\theoremstyle{plain}
\newtheorem{prop}{Proposition}[section]
\newtheorem{thm}[prop]{Theorem}
\newtheorem{lem}[prop]{Lemma}
\newtheorem{cor}[prop]{Corollary}

\newtheorem{ex}[prop]{Example}
\theoremstyle{definition}

\newtheorem{rem}[prop]{Remark}

\newif\ifSuppressEndOfProof\SuppressEndOfProoffalse

\makeatletter
\def\p@figure{Fig. }
\def\p@enumii{}
\makeatother

\newcommand{\cc}{\EuScript{C}}
\newcommand{\cd}{\EuScript{D}}
\newcommand{\MS}{\mathcal{MS}}

\newcommand{\sms}{\mathcal{S}\MS}
\newcommand{\F}{\mathrm{Facets}}
\newcommand{\link}{\mathrm{link}}
\newcommand{\supp}{\mathrm{supp}}
\newcommand{\rnd}{\partial}
\newcommand{\C}[1]{C$_#1$}
\newcommand{\N}{\mathrm{N}}
\newcommand{\G}{\mathcal{G}}

\newcommand{\z}{\mathbb{Z}}

\newcommand{\ifof}{if and only if }
\newcommand{\se}{\subseteq}

\newcommand{\sm}{\setminus }
\renewcommand{\iff}{\Leftrightarrow}
\newcommand{\give}{$\Rightarrow$}
\newcommand{\rgive}{$\Leftarrow$}

\newcommand{\rg}{\rangle}
\renewcommand{\lg}{\langle}
\newcommand{\p}[1]{\left(#1\right)}
\newcommand{\f}[2]{\frac{#1}{#2}}
\newcommand{\tohi}{\emptyset}
\renewcommand{\b}[1]{\overline{#1}}
\newcommand{\tl}[1]{\widetilde{#1}}

\newcommand{\blt}{\bullet}

\begin{document}

\title{On Generalizations of Cycles and Chordality to Hypergraphs from an Algebraic Viewpoint}
\author{Ashkan Nikseresht and Rashid Zaare-Nahandi \\
\it\small Department of Mathematics, Institute for Advanced Studies in Basic Sciences,\\
\it\small P.O. Box 45195-1159, Zanjan, Iran.\\
\it\small E-mail: ashkan\_nikseresht@yahoo.com\\
\it\small E-mail: rashidzn@iasbs.ac.ir}
\date{}
\maketitle

\begin{abstract}
In this paper, we study the notion of chordality and cycles in hypergraphs from a commutative algebraic point of
view. The corresponding concept of chordality in commutative algebra is having a linear resolution. However,
there is no unified definition for  cycle or chordality in hypergraphs  in the literature, so we consider
several generalizations of these notions and study their algebraic interpretations. In particular, we
investigate the relationship between chordality and having linear quotients in some classes of hypergraphs. Also
we show that if $\cc$ is a hypergraph such that $\lg \cc \rg$ is a vertex decomposable simplicial complex or
$I(\b\cc)$ is squarefree stable, then $\cc$ is chordal according to one of the most promising definitions.
\end{abstract}

Keywords and Phrases:  squarefree monomial ideal, linear resolution, chordal clutter, cycle, simplicial complex. \\
\indent 2010 Mathematical Subject Classification: 13F55, 05E45.

                                        \section{Introduction}
Let $\cc$ be a \emph{hypergraph} on the vertex set $[n]=\{1,\ldots,n\}$, that is, a family of subsets of $[n]$
called \emph{hyperedges}. If all hyperedges of $\cc$ have the same cardinality $d+1$, as we always assume in this
paper, we call $\cc$ a \emph{uniform $d$-dimensional clutter} (\emph{$d$-clutter}, for short). In this case, the
hyperedges of $\cc$ are called \emph{circuits}. Also throughout this paper, $S=k[x_1,\ldots, x_n]$, where $k$ is a
field and $x_F=\prod_{i\in F}x_i$, for $F\se [n]$. Moreover, $I(\cc)$ denotes the \emph{circuit ideal} of $\cc=
\lg x_F|F\in \cc\rg$ and $\b{\cc}$ is the $d$-\emph{complement} of $\cc$, which means the family of all
$(d+1)$-subsets of $[n]$ not in $\cc$.

We call $d$-subsets of circuits of $\cc$ \emph{maximal subcircuits (MS)} and denote the set of maximal subcircuits
by $\MS(\cc)$. By degree of a maximal subcircuit $e$, $\deg(e)$, we mean  the number of circuits of $\cc$
containing $e$. For $L\se [n]$, we write $\cc-L$ for the subclutter of $\cc$ consisting of those circuits which do
not contain $L$. For the sake of simplicity, we write for example $abc$ for the subset $\{a,b,c\}$ of $[n]$ or
$ev$ for $e\cup \{v\}$ where $e\se [n]$ and $v\in [n]$.

A well-known theorem of Fr\"oberg says that in the case $d=1$ (that is, when $\cc$ is a graph), $I(\b{\cc})$ has a
linear resolution \ifof $\cc$ is a chordal graph. Recently many authors have tried to generalize the concept of
chordal graphs to clutters of arbitrary dimension in a way that Fr\"oberg's theorem remains true for $d>1$, see
for instance  \cite{E-chordal,vv-chordal, chordal,w-chordal,CF1}. A key part of this generalization is how to
define cycles in arbitrary clutters.

One such generalization is given and studied in Connon and Faridi \cite{CF2, CF1}. They say that $\tohi \neq \cc$
is a $d$-dimensional cycle (which we call a \emph{$d$-dimensional CF-cycle}), when $\cc$ is \emph{strongly
connected} (that is, for each $F_1, F_2\in \cc$, there is a sequence of circuits starting with $F_1$ and ending
with $F_2$ such that consecutive circuits share a common MS) and $\deg(e)$ is even for each $e\in \MS(\cc)$. Also
by a \textit{$d$-dimensional CF-tree} we mean a nonempty uniform $d$-dimensional clutter without any CF-cycle.

Another approach to chordal graphs is the notion of simplicial vertex. A vertex is called simplicial if its
neighbourhood forms a clique. Dirac in \cite{Dir} proved that a graph $G$ is chordal if and only if there is a
sequence of vertices $v_1, \ldots, v_r$ such that $v_1$ is simplicial in $G$ and for each $1<i\leq t$, $v_i$ is
simplicial in $G-v_1-\cdots - v_{i-1}$. Some authors use this approach to define chordal clutters. R. Woodroof in
\cite{w-chordal} defines simplicial vertex in an arbitrary clutter and uses this notion to define chordal
clutters, which we will call W-chordal.

Motivated by a result of \cite{all}, another definition of chordal clutters is proposed by Bigdeli, et al
\cite{chordal}. In a $d$-dimensional uniform clutter, they define the notion of ``simplicial'' for $d$-subcircuits
instead of vertices. An MS $e$ of $\cc$ is said to be \textit{simplicial}, if   $\N[e]= e\cup\{v\in
[n]|e\cup\{v\}\in \cc\}$ is a \emph{clique} in $\cc$ (that is, for all $F\se \N[e]$ with $|F|=d+1$, we have $F\in \cc$).
We denote the set of all simplicial maximal subcircuits of $\cc$ by $\sms(\cc)$. If there exists a sequence
$e_1,\ldots, e_t$ with $e_i\in \sms\p{ \cc-e_1- \cdots - e_{i-1} }$ such that $\cc-e_1 -\cdots -e_t= \tohi$, then
$\cc$ is said to be \textit{chordal}. In particular, the empty clutter is considered chordal.

Remark 3.10 of \cite{all} states that if $\cc$ is a chordal clutter, then $I(\b{\cc})$ has a linear resolution
over every field. It still remains a question whether the converse is also true. It is known that if a graded
ideal $I$  of $S$ generated in one degree has linear quotients, then it has a linear resolution (see
\cite[Proposition 8.2.1]{hibi}). Thus as part of studying the converse of \cite[Remark 3.10]{all} mentioned above,
one can investigate whether $\cc$ is chordal, given that $I(\b \cc)$ has linear quotients. In this regards, we
present some relations between chordality and having linear quotients in Section 2. As a particular case, we prove
that if $I(\b\cc)$ is squarefree stable, then $\cc$ is chordal.

In Section 3, we try to find some classes of CF-trees which are chordal. We prove that if $\lg \cc \rg$ is a
vertex decomposable simplicial complex, then $\cc$ is a CF-tree which is also chordal. Furthermore, we prove that
every CF-tree $\cc$ on at most $\dim \cc+3$ vertices is chordal.

In Section 4, we study three other generalizations of cycles to clutters. Using these notions of cycles, we
present several generalizations of chordal graphs to clutters and study the relation of these generalizations with
the condition that $I(\b{\cc})$ has a linear resolution and also with chordality of $\cc$. Before stating the main
results, we present a brief review of simplicial complexes.

\subsection*{A brief review of simplicial complexes}

A \emph{simplicial complex} on vertex set $V$ is a family $\Delta$ of subsets of $V$ (called \emph{faces }of
$\Delta$) such that if $A\se B\in \Delta$, then $A\in \Delta$. We always assume that the vertex set is
$[n]$, unless specified otherwise. For each $F\in \Delta$, $\dim F=|F|-1$ and $\dim \Delta= \max_{F\in
\Delta}\dim F$. The set of maximal faces of $\Delta$ which are called \emph{facets} is denoted by $\F(\Delta)$. If
$|\F(\Delta)|=1$, then $\Delta$ is called a \emph{simplex}.

If all facets of $\Delta$ have the same dimension --- as we always assume in the sequel --- we say that $\Delta$
is pure. In this case $\F(\Delta)$ is a $d$-dimensional uniform clutter. Also if $\cc$ is a clutter, then $\lg \cc
\rg$ denotes the simplicial complex $\Delta$ with $\F(\Delta)= \cc$. Another simplicial complex associated to a
$d$-clutter $\cc$ on $[n]$ is the \emph{clique complex} $\Delta(\cc)$ of $\cc$ defined as the family of all
subsets $L$ of $[n]$ with the property that $L$ is a clique in $\cc$. Note that all subsets of $[n]$ with size
$\leq d$ are cliques by assumption.

The ideal generated by $\{x_F|F$ is a minimal non-face of $\Delta\}$ is called the \emph{Stanley-Reisner ideal} of
$\Delta$ and is denoted by $I_\Delta$. When $\f{S}{I_\Delta}$ is a Cohen-Macaulay ring, $\Delta$ is called
Cohen-Macaulay over $k$ or $k$-Cohen-Macaulay. For a face $F$ of $\Delta$ and a set $L\se [n]$, we define
$\link_\Delta F= \{G\sm F|G\in \Delta\}$ and $\Delta|_L =\{G\in \Delta| G\se L\}$. Moreover, we call $\Delta$
\emph{$d$-complete} when $\Delta$ has all $(d+1)$-subsets of $[n]$. We call a family $\cc$ of $k$-dimensional
faces of $\Delta$ a \emph{$k$-cycle}, if $\cc$ is a CF-cycle as a clutter.

Let $A$ be a commutative ring with identity and denote by $\tl{C}_d(\Delta)= \tl{C}_d(\Delta, A)$ the free
$A$-module whose basis is the set of all $d$-dimensional faces of $\Delta$. Consider the $A$-homomorphism $\rnd_d:
\tl{C}_d(\Delta) \to \tl{C}_{d-1}(\Delta)$ defined by
 $$\rnd_d(\{v_0, \ldots, v_d\}) =\sum_{i=0}^d (-1)^i \{v_0, \ldots, v_{i-1},v_{i+1},\ldots, v_d\},$$
where $v_0<\cdots <v_d$. Then $(\tl{C}_\blt, \rnd_\blt)$ is a complex of $A$-modules and $A$-homomorphisms called
the \emph{augmented oriented chain complex} of $\Delta$ over $A$. We denote the $i$-th homology of this complex by
$\tl{H}_i(\Delta; A)$.

By \emph{Alexander dual} of a simplicial complex $\Delta$ we mean $\Delta^\vee=\{[n]\sm F|F\se [n],\,F\notin
\Delta\}$ and also we set $\cc^\vee=\{[n]\sm F\big| F\in \b{\cc}\}$. Then it follows from the Eagon-Reiner theorem
(\cite[Theorem 8.1.9]{hibi}) and the lemma below that $I(\b{\cc})$ has a linear resolution over $k$, \ifof $\lg
\cc^\vee \rg$ is Cohen-Macaulay over $k$. For more details on simplicial complexes and related algebraic concepts
the reader is referred to \cite{hibi}. We frequently use the following easy lemma in the sequel without any
further mention.

\begin{lem}\label{transition}
Let $\cc$ be a $d$-clutter.
\begin{enumerate}
\item $I(\b\cc)=I_{\Delta(\cc)}$.
\item $\lg \cc^\vee \rg= (\Delta(\cc))^\vee$.
\end{enumerate}
\end{lem}
\begin{proof}
\ref{1}: Just note that $F\se [n]$ is a non-face of $\Delta(\cc)$ \ifof $F$ is not a clique in $\cc$ \ifof $C\se
F$ for some $C\in \b\cc$.

\ref{2}: We show that facets of $(\Delta(\cc))^\vee$ are circuits of $\cc^\vee$. By the note after \cite[Lemma
1.5.2]{hibi}, the facets of $(\Delta(\cc))^\vee$ are the complements (with respect to $[n]$) of minimal non-faces
of $\Delta(\cc)$. But according to the proof of \ref{1}, minimal non-faces of $\Delta(\cc)$ are exactly the
circuits of $\b\cc$ and their complements form $\cc^\vee$ by definition.
\end{proof}

                                     \section{Chordality and Linear Quotients}
In this section we investigate the relationship between chordality of  $\cc$ and having linear quotients for
$I(\b\cc)$. Example 3.15 of \cite{chordal}, shows that not for all chordal clutters $\cc$, the ideal $I(\b\cc)$
has linear quotients. Thus we focus on the question ``Are all clutters $\cc$ with $I(\b \cc)$ having linear
quotients, chordal?''. Our first result shows that this question is equivalent to asking ``Do  all clutters $\cc$
with $I(\b \cc)$ having linear quotients, have an SMS?'' Moreover, this theorem establishes an interesting
relationship between orders of linear quotients of certain ideals and chordality of clutters $\cc$ with $I(\b\cc)$
having linear quotients.

In the sequel, by a \emph{complete clutter} we mean a clutter in which the set of vertices is a clique. Also here
we say that a squarefree monomial ideal $I$ generated in degree $d$ is \emph{complete}, when all squarefree
monomials of $S$ with degree $d$ are in $I$ or equivalently if $I=I(\cc)$ for the complete $(d-1)$-clutter with
vertex set $[n]$.

\def\theenumi{\Alph{enumi}}
\begin{thm}\label{2 main}
Let $d>0$ be an integer and consider the following propositions.
\begin{enumerate}
\item \label{2 main 1}  Every $d$-clutter $\cc$ with $I(\b \cc)$ having linear quotients is chordal.

\item \label{2 main 2}  Every non-empty $d$-clutter $\cc$ with $I(\b \cc)$ having linear quotients has an SMS.

\item \label{2 main 3} Every order of linear quotients for a non-complete squarefree monomial ideal $I$ of $S$
    generated in degree $d$, can be extended to an order of linear quotients for the complete squarefree
    monomial ideal of $S$ generated in degree $d$.

\end{enumerate}
Then \ref{2 main 1} $\iff$ \ref{2 main 2} \give\ \ref{2 main 3}.
\end{thm}
\begin{proof}
\ref{2 main 1} \give\ \ref{2 main 2}: Trivial. \ref{2 main 2} \give \ref{2 main 1}: Assume that $I(\b\cc)$ has
linear quotients. We claim that for every $e\in \sms(\cc)$, the ideal $I(\b{\cc-e})$ has linear quotients. Then if
$\cc\neq \tohi$, by \ref{2 main 2}, $\cc-e$ has a SMS, say $e'$, and again $\cc-e-e'$ either is empty or has a SMS
by \ref{2 main 2} and hence we can continue deleting SMS's until we reach the empty clutter.

\emph{Proof of claim.} Let $I=I(\b \cc)$, $I'= I(\b{\cc-e})$ and suppose that $u_1, \ldots, u_m$ is the linear
quotient order  of the set of minimal generators of $I$, denoted by $\G(I)$. Then $\G(I')$ is the union of $\G(I)$
and the set of squarefree monomials $u$ of degree $d+1$ with $e\se \supp(u)=\{i\in [n]\big| x_i|u\}$. Denote the
elements of $\G(I')\sm \G(I)$ by $u_{m+1}, \ldots, u_t$. We show that $I'$ has linear quotients with respect to
the order $u_1, \ldots, u_t$. Set $F_i=\supp(u_i)$. By \cite[Corollary 8.2.4]{hibi}, we have to show that for each
$i$ and $j$ such that $j<i\leq t$, there are $l\in F_j\sm F_i$ and $k<i$ such that $F_k\sm F_i=\{l\}$. If $i\leq
m$, this property holds since $u_1, \ldots, u_m$ is an order of linear quotients. If $j\geq m+1$, then as $e\se
F_j\cap F_i$, we have $|F_j\sm F_i|=1$ and hence the required property again holds with $k=j$ and $l\in F_j\sm
F_i$.

So assume $j\leq m< i$. Suppose that for each $l\in F_j\sm e$, we have $el\in \cc$. Then $F_j\se \N[e]$ and since
$e$ is simplicial, we should have $F_j\in \cc$ which means $u_j\notin I(\b\cc)$, a contradiction. Therefore, there
is a $l\in F_j\sm e$ with $el\notin\cc$. So $el\in \b\cc$ and for some $k\leq m$, we should have $el=F_k$.
Consequently, this $k$ and $l$ satisfy the required property.

\ref{2 main 2} \give\ \ref{2 main 3}: Let $\cc$ be the complete $d$-clutter with vertex set $[n]$. We denote the
union of $\sms(\cc)$ with the set of $d$-subsets of $[n]$ not in any circuit of $\cc$ by $\sms'(\cc)$. So
$\sms'(\cc)$ is the set of all $d$-subsets of $[n]$ such that $\N[e]$ is a clique, whether $e\in \MS(\cc)$ or not.

Suppose that $x_{e_1}, x_{e_2}, \ldots, x_{e_t}$ is an order of linear quotients for the non-complete ideal $I=\lg
x_{e_1}, x_{e_2}, \ldots, x_{e_t}  \rg$, where $e_i$'s are $d$-subsets of $[n]$. We show that $e_i\in
\sms'(\cc_{i-1})$ for each $i\in [t]$ where $\cc_{i}=\cc-e_1-e_2-\cdots-e_i$. To see this, note that by
\cite[Corollary 8.2.4]{hibi},
 \centerline{\hfill for each $j<i$ there is a $k< i$ and $l\in [n]$ with $l\in e_j\sm e_i$ and
 $e_k\sm e_i=\{l\}$ \hfill ($*$)}
Therefore it follows that $e_k\se e_il$ and $l\notin \N_{\cc_{i-1}}[e_i]$. So $e_j\not\se \N_{\cc_{i-1}}[e_i]$ for
each $j<i$. Thus for each $F\se \N_{\cc_{i-1}}[e_i]$ with $|F|=d+1$, $F\in {\cc_{i-1}}$ and $\N_{\cc_{i-1}}[e_i]$
is a clique, as required.

Now since $I(\b\cc)$ has linear quotients, it follows from the proof of \ref{2 main 2} \give\ \ref{2 main 1} above
that $I(\b\cc_t)$ has linear quotients, too (note that if $e_i\in \sms'(\cc_{i-1})\sm \sms(\cc_{i-1})$, then
$\cc_i= \cc_{i-1}$). If $\cc_t\neq \tohi$, let $e_{t+1}\in \sms(\cc_t)$, which is not empty by \ref{2 main 2}. If
$\cc_t=\tohi$, then as $I$ is not complete, there is a $d$-subset $e_{t+1}$ of $[n]$ such that $x_{e_{t+1}}\notin
I$. Note that in both cases $e_{t+1}\in \sms'(\cc_t)$. Therefore $e_i\in \sms'(\cc_{i-1})$ for each $i\in [t+1]$
where $\cc_{i}=\cc-e_1-\cdots-e_i$. We show that $x_{e_1}, \ldots, x_{e_{t+1}}$ is an order of linear quotients
for $I'= \lg x_{e_1}, \ldots, x_{e_{t+1}}  \rg$.

Suppose $j<i$. If $e_j\se \N_{\cc_{i-1}}[e_i]$, then for $y\in e_i\sm e_j$, we have $F=e_jy\notin \cc_{i-1}$ but
$F\se \N_{\cc_{i-1}}[e_i]$, contradicting simpliciality of $e_i$. Thus there exists an $l\in e_j\sm
\N_{\cc_{i-1}}[e_i]$. Hence $e_il$ is not a face of $\cc_{i-1}$, that is, there is a $k< i$ with $e_k\se e_il$.
Consequently, $e_k\sm e_i=\{l\}$. Therefore $(*)$ above holds and the claim follows from \cite[Corollary
8.2.4]{hibi}.

If $I'$ is not complete, then by the same argument we can extend the order of linear quotients of $I'$ to one for
$I'+\lg x_{e_{t+2}} \rg $ for some $d$-subset $e_{t+2}$ of $[n]$. Continuing this way, in each step we extend this
order of linear quotients by one, until we reach the complete ideal and the result follows.
\end{proof}

\def\theenumi{\roman{enumi}}
Observe that the following statements were indeed proved in the proof of the previous theorem.
\begin{rem}\label{lin q and sms seq}
Let $\cc$ be a $d$-clutter, $E$ be the set of $d$-subsets of $[n]$ and $\sms'(\cc)$ be as defined in the proof of
\ref{2 main}.
\begin{enumerate}
\item \label{1} Suppose that $I(\b \cc)$ has linear quotients and that $e\in \sms(\cc)$. Then $I(\b{\cc-e})$ has
    linear quotients.
\item \label{2} Assume that $\cc$ is the complete $d$-clutter with vertex set $[n]$. For a sequence $e_1, e_2,
    \ldots, e_t$ of different elements in $E$ the following are equivalent:
\begin{enumerate}
\item $e_i\in \sms'(\cc_{i-1})$ for each $i\in [t]$, where $\cc_{i}=\cc-e_1-e_2-\cdots-e_i$.

\item $x_{e_1}, x_{e_2}, \ldots, x_{e_t}$ is an order of linear quotients for $\lg x_{e_1}, x_{e_2}, \ldots,
    x_{e_t}  \rg$.
\end{enumerate}

\end{enumerate}
\end{rem}

It should be noted that \ref{lin q and sms seq} provides another proof of \cite[Lemma 3.11]{chordal} which states
that every complete clutter is chordal.

Due to the above result, we try to find some subclasses of ideals with linear quotients whose corresponding
clutter has an SMS. One such class is the set of squarefree polymatroidal ideals. A monomial ideal $I$ is called
\emph{polymatroidal}, when all elements  of $\G(I)$ have the same degree and if $u=x^{a_1} \cdots x^{a_n}, v=
x^{b_1} \cdots x^{b_n} \in \G(I)$ with $a_i>b_i$ for some $i$, then there is a $j$ with $a_j<b_j$ such that
$x_ju/x_i\in \G(I)$. According to the symmetric exchange theorem (\cite[Theorem 12.4.1]{hibi}), when $I$ is
polymatroidal and $u,v,i$ are as above, then we can find a $j$ with $a_j<b_j$ and not only $x_ju/x_i\in \G(I)$ but
also $x_iv/x_j\in \G(I)$. Also by \cite[Theorem 12.6.2]{hibi}, polymatroidal ideals have linear quotients.

\begin{prop}
Suppose that $\cc\neq \tohi$ and $I(\b\cc)$ is polymatroidal. Then $\sms(\cc)\neq \tohi$.
\end{prop}
\begin{proof}
If $\cc$ is complete then by \cite[Corollary 3.12]{chordal}, $\cc$ is chordal. Thus we assume that $\b\cc\neq
\tohi$. Obviously if $\cc'$ is the complete $d$-clutter on $[n]$, then $\cc'$ is strongly connected. Thus as $\cc
\cup \b\cc=\cc'$, we conclude that there is an $e\in \MS(\cc)\cap \MS(\b\cc)$. We show that $e\in \sms(\cc)$.

Assume $A\se \N_\cc[e]$, $A\in \b\cc$ and $|A\cap e|$ is maximum possible. If $|A\cap e|=d$, then $A=ev$ for some
$v\in \N_\cc[e]$, a contradiction. Suppose $|A\cap e|<d$. Since $e\in \MS(\b\cc)$, there is a $F\in \b\cc$
containing $e$. Let $i\in e\sm A$. Then by the  $I(\b\cc)$ being polymatroidal and by the symmetric exchange
theorem, we deduce that there exists a $j\in A\sm F$ such that $A'= A\sm \{j\}\cup \{i\}\in \b\cc$. Note that
$j\notin e$, so $|A' \cap e|> |A\cap e|$ and also $A'\se A\cup e\se \N_\cc[e]$. This is against the choice of $A$.
We conclude that no such $A$ exists, that is, $\N_\cc[e]$ is a clique, as required.
\end{proof}

We cannot deduce from this proposition that if $I(\b\cc)$ is polymatroidal, then $\cc$ is chordal, since
$I(\b{\cc-e})$ may not be polymatroidal for any $e\in \sms(\cc)$, as the following example shows.
\begin{ex}
Let $\cc=\{145, 245,345\}$ on $[5]$. Then it is straightforward to check that $I(\b\cc)$ is polymatroidal. Note
that $\sms(\cc)=\{ab|a=1,2,3, \, b=4,5\}$. Thus for every $e\in\sms(\cc)$, $I(\b{\cc-e})$ is not polymatroidal.
For example if $e=14$, $F_1=145$ and $F_2=234$, then $F_1, F_2\in \b{\cc-e}$ and $1\in F_1\sm F_2$ but there is no
$j\in F_2\sm F_1$ with $F_1\sm \{1\} \cup\{j\}\in \b{\cc-e}$.  Despite this, $e'=34\in \sms(\cc-e)$, $e''=24\in
\sms(\cc-e-e')$, $\cc-e-e'-e''=\tohi$, so that $\cc$ is chordal.
\end{ex}

For a monomial $u$ let $m(u)$ denote the largest $i$ with $x_i|u$. Another class of ideals with linear quotients
are \emph{squarefree stable ideals}, that is, squarefree monomial ideals such as $I$ with the property that for
each squarefree monomial $u\in I$ and for each $j<m(u)$ with $x_j \not| u$ one has $x_ju/x_{m(u)}\in I$. Next we
show that if $I(\b\cc)$ is squarefree stable, then $\cc$ has an SMS and even more, $\cc$ is chordal.

\begin{thm}\label{sq f sta}
Assume that $I(\b\cc)$ is squarefree stable, then $\cc$ is chordal.
\end{thm}
\begin{proof}
Since the empty clutter is chordal by definition, we can assume that $\cc\neq \tohi$. For each $1\leq i\leq d$,
let $a_i$ be the minimum number $\neq a_1, \ldots, a_{i-1}$ in $[n]$, for which there is a $F\in \cc$ with $a_1,
\ldots, a_i\in F$. Set $e=a_1\cdots a_d$. Then $e\in \MS(\cc)$. We show that $e\in \sms(\cc)$ and $I(\b{\cc-e})$
is squarefree stable. Then the result follows by induction.

Suppose that $F\se \N[e]$, $F\in \b\cc$ and $|F\cap e|$ is maximum possible. By the way we chose $a_i$'s, we know
that for each $b\in \N[e]\sm e$, we have $b>a_d>\cdots> a_1$, in particular, $m(x_F)\notin e$. If $|F\cap e|=d$,
then $F=eb$ for some $b\in \N[e]$, that is, $F\in \cc$ a contradiction. Thus we can assume that $|F\cap e|<d$.
Hence $e\sm F\neq \tohi$. Set $b'=\min e\sm F$. Then as $x_F\in I(\b\cc)$ which is squarefree stable, we deduce
that $F'= (F\sm \{m(x_F)\})\cup \{b'\}\in \b\cc$. Note that $F'\cap e$ strictly contains $F\cap e$ and as $b'\in
e$, we have $F'\se F\cup\{b'\}\se \N[e]$. So by maximality of $|F\cap e|$, we should have $F'\in \cc$, a
contradiction. From this contradiction we conclude that $F\in \cc$ and hence $e\in \sms(\cc)$.

Next suppose that $F\in \b{\cc-e}$ with $b=m(x_F)$ and let $b'<b$ with $b'\notin F$ and $F'=F\sm \{b\}\cup
\{b'\}$. We should prove that $F'\in \b{\cc-e}$, to show that $I(\b{\cc-e})$ is squarefree stable. If $F\notin
\cc$, then $F\in \b\cc$ and it follows that $F'\in \b\cc\se \b{\cc-e}$. If $F\in \cc$, then we should have $e\se
F$ and hence $b\notin e$. Consequently, $e\se F'$ which results to $F'\notin \cc-e$. Hence $F'\in \b{\cc-e}$ and
the result is concluded.
\end{proof}

Recall that a squarefree monomial ideal $I$ is called \emph{squarefree strongly stable}, when for each squarefree
monomial $u\in I$ and for each $j<i$ with $x_j \not| u$ and $x_i|u$ one has $x_j u/x_i\in I$. Clearly each
squarefree strongly stable ideal is squarefree stable and hence we can apply the previous result on such ideals,
too. But indeed, we can say more on chordality of squarefree strongly stable ideals. To see this, we need to
recall another concept which generalizes chordal graphs.

Let $\cd$ be a clutter on $[n]$ which is not necessarily uniform (that is the $\cd$ is a family of incomparable
subsets of $[n]$ which may have different sizes). For $v\in [n]$ by $\cd/v$ (\emph{contraction} on $v$) we mean
the clutter of minimal elements of $\{F\sm \{v\}|F\in \cd\}$. In \cite{w-chordal}, a vertex $v$ of $\cd$ is called
\emph{simplicial}, when from $v\in F_1, F_2\in \cd$ we can deduce that there is a $F_3\in \cd$ with $F_3\se
(F_1\cup \F_2)\sm \{v\}$. Moreover in \cite{w-chordal},  Woodroofe calls $\cd$ chordal (here we call $\cd$
\emph{W-chordal}), when every clutter arising from $\cd$ by a series of vertex deletions and contractions,
contains a simplicial vertex.

\begin{prop}
Assume that $I(\cc)$ is squarefree strongly stable. Then $\cc$ is both chordal and W-chordal.
\end{prop}
\begin{proof}
More generally, let $\cd$ be a (not necessarily uniform) clutter with $I(\cd)$ squarefree strongly stable. It is
easy to show that for any $v\in [n]$, $I(\cd-v)$ is squarefree strongly stable. We first show that $I(\cd/v)$ is
also strongly stable. Suppose that $F\in \cd/v$, $i\in F$, $j\notin F$ and $j<i$. We should prove that there is a
$F'\in \cd/v$ with $F'\se F\cup \{j\}\sm \{i\}$. If $F\in \cd$, by $I(\cd)$ being squarefree strongly stable,
there is a $F''\in \cd$ satisfying this containment and so there is a $F'\in \cd/v$ with $F'\se F''$ and we are
done. If $F\notin \cd$, then $F\cup\{v\}\in \cd$, so there is a $F''\in \cd$ with $F''\se F\cup \{j,v\}\sm \{i\}$.
Again there is a $F'\in \cd/v$ with $F'\se F''\sm\{v\}$, as required.

It is easy to see that if $m$ is the maximum integer appearing in the union of all circuits of $\cd$, then $m$ is
a simplicial vertex. Thus $\cd$ is W-chordal. Furthermore, one could easily check that $I(\b\cd)$ is squarefree
strongly stable with respect to the reverse order on $[n]$. Consequently, if $\cd$ is also uniform, then by
\ref{sq f sta}, $\cd$ is chordal.
\end{proof}
It should be mentioned that in the above result, chordality of $\cc$ also follows its W-chordality, according to
\cite[Proposition 3.8]{chordal}.


Next we investigate, ideals with linear quotients which are the vertex cover ideal of some graph. Recall that the
\emph{vertex cover ideal} of a graph $G$ on $[n]$ is the ideal generated by all $x_C$'s where $C$ is a minimal
vertex cover of $G$ and as in \cite{hibi}, we denote this ideal by $I_G$. In the next result, note that
$\F(\Delta(\b G))$ is the clutter whose circuits are all maximal cliques of $\b G$ or equivalently all maximal
independent sets of $G$.

\begin{thm}\label{ver cover ideal}
Suppose that $\cc\neq \tohi$ and $I(\b\cc)=I_G$ for a graph $G$ with the property that $\F(\Delta(\b G))$ is
strongly connected. Then $\sms(\cc)\neq \tohi$.
\end{thm}
\begin{proof}
If $G$ is complete, then every $(n-1)$-subset of $[n]$ is a minimal vertex cover and it follows that $\cc=\tohi$.
Thus $G$ is not complete. If $G$ has exactly one maximal independent set, then it easily follows that $E(G)=\tohi$
which means $\cc$ is complete and $\sms(\cc)\neq \tohi$. Thus we can assume that $G$ has at least two maximal
independent sets.  Note that as $\cc$ is assumed to be $d$-uniform, all minimal vertex covers and hence all
maximal independent sets of $G$ have the same size.

Since facets of $\Delta(\b G)$ are the maximal independent sets of $G$ and because $\F(\Delta(\b G))$ is strongly
connected, there exist two maximal independent sets $A_1$ and $A_2$ of $G$ with $|A_1\cap A_2|=|A_1|-1= |A_2|-1$.
Set $B=A_1\cap A_2$ and suppose that $v_1\in A_1\sm B, v_2 \in A_2\sm B$. Consider $e=[n]\sm (A_1 \cup A_2)$. Then
$ev_1$ and $ev_2$ are minimal vertex covers of $G$. If $v\in \N_{\b\cc}[e]\sm e$ and $v\neq v_1, v_2$, then $ev$
is a vertex cover of $G$ not containing any of $v_1,v_2$. Therefore, $v_1v_2$ is not an edge of $G$. On the other
hand, since $A_1\cup A_2$ is not an independent set and as $(A_1\cup A_2 )\sm (A_1 \cap A_2)=\{v_1,v_2\}$, we
conclude that $v_1v_2$ is an edge. From this contradiction, it follows that $\N_{\b\cc}[e]=ev_1v_2$.

If $ev_1v_2=[n]$, then $A_1\cap A_2=\tohi$ and hence maximal independent sets of $G$ are verieces, which means $G$
is complete, against our assumption. Hence $\N_\cc[e]\neq \tohi$ and $e\in \MS(\cc)$. Let $A\se N_\cc[e]$ and
$|A|=d+1$. If $A\in \b\cc$, then $A$ is a minimal vertex cover of $G$ and as $v_1v_2\in E(G)$, one of $v_1$ and
$v_2$, say $v_1$, is in $A$. But $v_1\notin \N_\cc[e]$, a contradiction. Thus $A\in \cc$ and $e\in \sms(\cc)$.
\end{proof}

\begin{cor}
Suppose that $\cc\neq \tohi$, $I(\b\cc)=I_G$ for a graph $G$ and $I_G$ has a linear resolution (for example if
$I_G$ has linear quotients). Then $\sms(\cc)\neq \tohi$.
\end{cor}
\begin{proof}
Since $I_G=I(G)^\vee=I_{\Delta(\b G)^\vee}$ has a linear resolution, it follows from the Eagon-Reiner theorem that
$\Delta(\b G)$ is a Cohen-Macaulay complex. Hence according to \cite[Lemma 9.1.12]{hibi}, $\F(\Delta(\b G))$ is
strongly connected. Therefore, the result follows from \ref{ver cover ideal}.
\end{proof}

Despite the above result, we cannot deduce that in the situations of this result $\cc$ is chordal, because as the
following results show, the class of vertex cover ideals is not closed under removing SMS's. A graph is said to be
\emph{unmixed}, when all of its minimal vertex covers have the same size.
\begin{prop}\label{ver cov ideal = sc deleted}
Suppose that $e\in \MS(\cc)$ and $\G(I(\b{\cc-e}))\se \G(I_G)$ for an unmixed graph on $[n]$. Then $n=d+2$ and $G$
is a complete graph. In particular, $I(\b{\cc-e})= I_G$ for some graph $G$, \ifof $\dim \cc=n-2$ and $\lg \cc \rg$
is either a simplex or the union of two simplexes sharing an MS.
\end{prop}
\begin{proof}
Note that $\b{\cc-e}$ contains all circuits which contain $e$, therefore each such subset of $[n]$ is a minimal
vertex cover of $G$. Choose $v_1\in [n]\sm e$. Then as $ev_1$ is a minimal vertex cover, there is an edge $v_1v_2$
of $G$ with $v_2\notin e$. Also each edge whose both ends are outside $e$, should have $v_1$ as one end and
similarly $v_2$ as the other end. Thus $[n]=ev_1v_2$, and $n=d+2$. If $G$ is not complete, say $xy\notin E(G)$ for
some $x,y\in [n]$, then $G$ has a minimal vertex cover of size $<n-1$ contained in $[n]\sm xy$. This contradicts
unmixedness of $G$, so $G$ is complete.

For the in particular case, first note that since $\cc$ is uniform, $G$ should be unmixed. Also because every
non-isolated vertex of a graph lies in a minimal vertex cover and by deleting isolated vertices of $G$, we can
assume that $G$ has the same vertex set as $\cc$. Thus by the general part, we get $\dim \cc=n-2$ and that
$\b{\cc-e}$ is a complete $d$-clutter. But all circuits of $\b{\cc-e}$ which does not contain $e$ are in $\b\cc$.
It follows that all circuits of $\cc$ contain $e$. Therefore, if $[n]\sm e=v_1v_2$, then $\cc$ is either
$\{ev_1\}$ or $\{ev_2\}$ or $\{ev_1, ev_2\}$, as required. The converse is clear.
\end{proof}
\begin{cor}
Suppose that $I(\b\cc)=I_G$ for some graph $G$ and $e\in \sms(\cc)$. Then $I(\b{\cc-e}))\neq I_{G'}$ for any graph
$G'$.
\end{cor}
\begin{proof}
If $I(\b{\cc-e}))= I_{G'}$ for some graph $G'$, then by \ref{ver cov ideal = sc deleted} and using the fact that
$e\in \sms(\cc)$, it follows that $\cc$ should be a $(n-2)$-dimensional simplex, say $\cc=\{ev_1\}$ where $[n]\sm
e=v_1v_2$. Thus $ev_2\in I_G$ is a minimal vertex cover of $G$. So there is an edge in $G$ with one end $v_2$ and
the other end outside of $e$. As in the proof of the above result, we should have $V(G)=[n]$, thus $v_1v_2\in
E(G)$. From this it follows that $ev_1$ is also a minimal vertex cover of $G$, that is, $ev_1\in \b\cc$, a
contradiction from which the result follows.
\end{proof}

                                     \section{Chordality and CF-Trees}

In this section, we consider the question ``which CF-trees are chordal?'' First we prove that the facets of a
vertex decomposable simplicial complex form a CF-tree which is chordal. Recall that a pure simplicial complex
$\Delta$ is called \emph{vertex decomposable}, when whether it is a simplex or there is a $v\in V(\Delta)$ such
that both $\link_\Delta v$ and $\Delta\sm v (=\Delta|_{V\sm\{v\}})$ are both pure vertex decomposable. We call the
vertex $v$ a \emph{shedding vertex}. This concept was first introduced in \cite{provan} in connection with the
Hirsch conjecture which has applications in the analysis of the simplex method in linear programming and later was
studied by other authors, for example see \cite{obs to shell, bal ver dec}. Note that 0-dimensional vertex
decomposable simplicial complexes are $\lg \{v\} \rg$ for a vertex $v$.

\begin{lem}\label{1-dim vdec = tree}
Suppose that $\cc$ is a 1-dimensional uniform clutter (that is, a graph). Then $\Delta= \lg \cc \rg$ is vertex
decomposable \ifof $\cc$ is a tree. Also if $\Delta$ is not a simplex, then $v$ is a shedding vertex \ifof it is a
free vertex (that is, a vertex with degree one).
\end{lem}
\begin{proof}
(\rgive): If $|\cc|=1$, then $\cc$ is a simplex and vertex decomposable. Thus assume that $|\cc|>1$. Let $v$ be a
free vertex of $\cc$ and $l$ the vertex adjacent to $v$. Then $\link_\Delta v=\lg \{l\} \rg$ is vertex
decomposable and because $\cc$ is connected, $l$ is adjacent to a vertex $\neq v$. So $\{l\}\notin \F(\Delta\sm
v)$, that is, $\Delta\sm v$ is pure and $\F(\Delta \sm v)$ is a tree with fewer vertices. Hence the result follows
by induction.

(\give): Let $v$ be a shedding vertex of $\Delta$. Since $\link_\Delta v$ is 0-dimensional vertex decomposable, we
see that $\link_\Delta v=\lg \{u\} \rg$ for a vertex $u$ and hence $\deg(v)=1$. Again since $\Delta\sm v$ is
vertex decomposable with fewer vertices, the claim follows by induction.
\end{proof}

Suppose that $\Delta=\lg \cc \rg$ and $v\in [n]$. It could be possible that $\lg \cc-v \rg \neq \Delta\sm v$, for
example if $\cc=\{12,34\}$, then $\cc-1=\{34\}$ but $\F(\Delta \sm 1)=\{34, 2\}$. But we have
\begin{lem}\label{v-deletion}
Assume that $\Delta=\lg \cc \rg$ is vertex decomposable and $v$ is a shedding vertex. Then if $\cc-v\neq \tohi$,
we have $\lg \cc-v \rg = \Delta\sm v$.
\end{lem}
\begin{proof}
It is clear that $\cc-v \se \F(\Delta \sm v)$. Let $F\in \F(\Delta \sm v)$. Since $\tohi\neq \cc-v\se \F(\Delta\sm
v)$, and because $\Delta\sm v$ is pure, we conclude that $\dim F=d$. So $v\notin F\in \F(\Delta)=\cc$ and the
claim follows.
\end{proof}

Here we call $e\in \MS(\cc)$ a \emph{free}, when $\deg(e)=1$. In this case $\N[e]$ is exactly a circuit of $\cc$
and hence a clique. Therefore every free MS is a SMS.
\begin{thm}\label{vdec has leaf}
Suppose that $\Delta= \lg \cc \rg$ is vertex decomposable and not a simplex and $v$ is a shedding vertex. Then
there exists a free MS $e$ of $\cc$ containing $v$ such that either $\cc-e=\cc -v$ or $\lg \cc-e \rg$ is vertex
decomposable and $v$ is a shedding vertex of $\lg \cc-e \rg$.
\end{thm}
\begin{proof}
We prove the result by induction on $d$. The case $d=1$ follows from \ref{1-dim vdec = tree}. Assume that $d>1$.
Set $\Delta'= \link_\Delta v$ and $\cc'=\F(\Delta')$. If $\Delta'$  is a simplex let $e'$ be any SMS of $\cc'$.
Else let $v'$ be a shedding vertex of $\Delta'$ and $e'$ be the SMS containing $v'$ provided by the induction
hypothesis. Applying \ref{v-deletion} on $\cc'$ and $v'$ and using the induction hypothesis, we see that in both
cases, $e'$ is a free MS of $\cc'$ and either $\lg \cc'-e' \rg$ is vertex decomposable or $\cc'-e'=\tohi$. Set
$e=e'v$. Then $\deg_\cc[e]=\deg_{\cc'}[e']=1$. If $\cc'-e'=\tohi$, then $\cc-e=\cc-v$. Assume $\cc'-e'\neq \tohi$
and let $\Gamma'=\lg \cc'-e' \rg$ and $\Gamma=\lg \cc-e \rg$. Then $\link_\Gamma v=\Gamma'$ which is vertex
decomposable. If $\cc-v\neq \tohi$, then by \ref{v-deletion} we get
$$ \Delta\sm v=\lg \cc-v \rg= \lg \cc-e-v \rg \se \Gamma\sm v$$
and as $\Gamma\se \Delta$, it follows that $\Gamma\sm v=\Delta\sm v$. Hence $\Gamma\sm v$ is vertex decomposable.
If $\cc-v=\tohi$, then every facet of $\Gamma$ contains $v$ and hence $\Gamma\sm v= \link_\Gamma v$ is again
vertex decomposable. Hence $v$ is a shedding vertex of $\Gamma$ and $\Gamma$ is vertex decomposable, as required.
\end{proof}

\begin{cor}
Suppose that $\Delta= \lg \cc \rg$ is vertex decomposable. Then $\cc$ is both a CF-tree and chordal. In
particular, if $\Delta$ is a pure $d$-dimensional vertex decomposable simplicial complex and $\Delta'$ is obtained
from $\Delta$ by adding all faces of dimension $< d$, then $I_{\Delta'}$ has a linear resolution.
\end{cor}
\begin{proof}
Suppose that $\Delta$ has a shedding vertex $v$ and $e_1$ is the free MS of $\cc$ containing $v$ obtained by
\ref{vdec has leaf}. Then either $\cc-e_1=\cc-v$ or $v$ is still a shedding vertex in $\cc-e_1$. In the latter
case by applying \ref{vdec has leaf} again, we can find a free MS $e_2$ of $\cc_2=\cc-e_1$ such that either
$\cc-e_2=\cc-v$ or $v$ is a shedding vertex of $\cc-e_2$. Therefore, by repeatedly applying \ref{vdec has leaf},
we can find a sequence of $e_1, \ldots, e_t$ of subsets of $[n]$ containing $v$ such that $e_i$ is a free MS of
$\cc_{i}=\cc-e_1-\cdots -e_{i-1}$ for each $i$, and $\cc_{t+1}= \cc-v$. If $\cc_{t+1}$ is a simplex or
$\cc_{t+1}=\tohi$, then $\cc_{t+1}$ and hence $\cc$ are chordal. Thus we can assume that $\cc_{t+1}= \cc-v \neq
\tohi$. Then by \ref{v-deletion}, we have $\lg \cc-v \rg=\Delta\sm v$ is vertex decomposable with fewer vertices.
Consequently, by induction it follows that $\cc$ is chordal.

To show that $\cc$ is a CF-tree, note that as all $e_i$'s found above are free MS'es, we have indeed proved that
there is a sequence $e_1, \ldots, e_k$ such that $e_i$ is a free MS of $\cc-e_1\cdots-e_{i-1}$ and $\cc-
e_1\cdots- e_k= \tohi$. Now assume that $\cc'\se \cc$ is a CF-cycle. As every MS of $\cc'$ has even degree, we see
that $e_1\notin \MS(\cc')$. So $\cc'\se \cc-e_1$. Whence by a similar argument $e_2\notin \MS(\cc')$ and $\cc'\se
\cc-e_1-e_2$. Continuing this way, we see that in fact $\cc'=\tohi$, which means, $\cc$ has no CF-cycles.

For the in particular case, just note that if $\cc=\F(\Delta)$, then $I(\b\cc)= I_{\Delta(\cc)}$ and since $\cc$
is a CF-tree, it follows that $\Delta(\cc)= \Delta'$.
\end{proof}


An example of a CF-tree which is not chordal is presented in \cite[p. 17]{Connon}, which is a triangulation of the
mod 3 Moore space (see \oldref{fig2}). But constructing such examples is hard. A question that may arise is that
for which $n$ and $d$ we can find a non-chordal $d$-dimensional CF-tree on $n$ vertices. Considering this question
we have:

\begin{figure}[!t]
\begin{center}
\includegraphics[scale= 0.8]{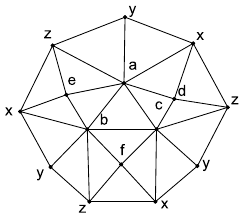}
\end{center}
\caption{A triangulation of the mod 3 Moore space; the circuits are the small triangles} \label{fig2}
\end{figure}

\begin{lem}\label{sms is leaf}
Let $\cc$ be a CF-tree. Then $e\in\sms(\cc)$ \ifof\ $e$ is a free.
\end{lem}
\begin{proof}
($\Leftarrow$): Obvious. (\give): If $\deg(e)>1$, then $|\N[e]|\geq d+2$. As $\N[e]$ is a clique, we conclude
that $\cc$ has a clique of size $d+2$. But one can readily check that every clique of size $d+2$ is a
CF-cycle, contradicting the assumption. So $\deg(e)=1$, as claimed.
\end{proof}

\begin{cor}\label{trees have leaf}
The following statements are equivalent.
\begin{enumerate}
\item Every $d$-dimensional CF-tree on $[n]$ is chordal.

\item Every $d$-dimensional CF-tree on $[n]$ has a free MS.

\item \label{3} Every $d$-dimensional clutter $\cc$ on $[n]$ with $\deg(e)>1$ for all $e\in\MS(\cc)$, has a
    CF-cycle.
\end{enumerate}
\end{cor}
\begin{proof}
\ref{2} $\Leftrightarrow$ \ref{3}: Obvious. \ref{1} \give\ \ref{2}: Clear by the previous lemma. \ref{2} \give\
\ref{1}: Note that every subclutter of every CF-tree is again a CF-tree. Thus if \ref{2} holds and $e_1$ is a free
MS of a CF-tree $\cc$, then $\cc-e_1$ has a free MS $e_2$, $\cc-e_1-e_2$ has a free MS $e_3$, \ldots. These
$e_i$'s satisfy the definition of chordality for $\cc$.
\end{proof}

Next we utilize Alexander duality to get conditions equivalent to the ones in the above result. For this we need a
concept which we call a CF-chorded clutter. Using the notion of CF-cylces, Connon and Faridi defined
\textit{chorded simplicial complexes} (see \cite[Definitions 8.2 and 4.2]{CF1}). We say that $\cc$ is
\emph{CF-chorded}, when the clique complex $\Delta(\cc)$ is a chorded simplicial complex in the sense of
\cite[Definition 8.2]{CF1} of Connon and Faridi. Here we do not state the exact definition. What is important for
us is that $I(\b{\cc})$ has a linear resolution over any field of characteristic 2 \ifof $\cc$ is CF-chorded (see
(a) $\iff$ (b) $\iff$ (d) of \cite[Theorem 18]{CF2}). In particular, CF-trees are CF-chorded and hence have linear
resolution over fields of characteristic 2 (see \cite[Theorem 7.3]{CF1}).

Also it follows immediately from the definition of a CF-chorded clutter such as $\cc$ that for any CF-cycle $C\se
\cc$ which is not a clique, there is another CF-cycle $C'\se \cc$ whose vertex set is a strict subset of the
vertex set of $C$ (see \cite[Defenitions 8.2, 4.2 and 4.1]{CF1}). In particular, $\cc$ has a clique on a
($d+2$)-subset of the vertex set of any CF-cycle in $\cc$.

We call $\{e\in \MS(\cc)|\deg(e)$ is odd$\}$ the \emph{boundary of $\cc$} and denote it by $\rnd(\cc)$. Also we
say that a simplicial complex $\Delta$ is \emph{almost $d$-complete} when $\Delta$ contains all $(d+1)$-subsets of
$[n]$ except exactly one of them.
\begin{thm}\label{dual}
Assume $d\leq n-2$ and let $d'=n-d-2$. The following are equivalent.
\begin{enumerate}
\item Every $d$-dimensional CF-tree on $[n]$ is chordal.

\item For every  non-$d'$-complete $d'$-dimensional $\z_2$-Cohen-Macaulay simplicial complex $\Delta$ on $[n]$
    which is $(d'-1)$-complete, there is a $(d'+2)$-set $L\se [n]$ such that $\Delta|_L$ is almost $d$-complete.

\item For every non-$d'$-complete $d'$-dimensional simplicial complex $\Delta$ on $[n]$ with the following
    properties, there is a $(d'+2)$-set $L\se [n]$ such that $\Delta|_L$ is almost $d$-complete.
\begin{enumerate}
\item \label{-1} $\Delta$ is $(d'-1)$-complete,

\item \label{-2} for every $A\se [n]$ with $|A|\geq n-d'+1$, every $k$-cycle with vertex set $A$ is the boundary
of a family of facets of $\link_\Delta([n]\sm A)$ where $k=|A|-(n-d'+1)$.
\end{enumerate}
\end{enumerate}
\end{thm}
\begin{proof}
\ref{1} \give\ \ref{2}: Suppose that $\Delta$ is a non-complete $d'$-dimensional $\z_2$-Cohen-Macaulay simplicial
complex on $[n]$ which contains all possible faces with dimension = $d'-1$. Set $\cc'=\F(\Delta)$ and
$\cc=\cc'^\vee$. Then by the Eagon-Reiner theorem, $I(\b{\cc})$ has a linear resolution over $\z_2$. Hence by
\cite[Theorem 18]{CF2} $\cc$ is CF-chorded. Now assume that $\cc$ has a CF-cycle. Then as noted in the remarks
before the theorem, $\cc$ should contain a clique on a set $L$ of $d+2$ vertices. Let $\b{L}$ denote $[n]\sm L$.
Then $\dim \b{L}=d'-1$ and by assumption, $\b{L}\in \Delta$. Whence there is a $F\in \cc'$ such that $\b{L}\se F$.
So $\b{F}\notin \cc$ and $\b{F}\se L$, which is in contradiction with $L$ being a clique.

Consequently, $\cc$ has no CF-cycles and is a $d$-dimensional CF-tree. Thus by \ref{1} and \ref{trees have leaf},
$\cc$ has a free MS, say $e$. Set $L=\b{e}$, then $|L|=n-d=d'+2$. Now a $(d'+1)$-subset $F\se L$ is a facet of
$\Delta$ \ifof $e\se \b{F}$ is not a circuit in $\cc$. As $e$ is a free MS and there is just one circuit
containing $e$, we conclude that $\Delta|_L$ is almost d-complete.

\ref{2} \give\ \ref{1}: Suppose that $\cc$ is a $d$-dimensional CF-tree. Then $\cc$ is CF-chorded and hence by
\cite[Theorem 18]{CF2}, $I(\b{\cc})$ has a linear resolution over $\z_2$ and $\Delta=\lg \cc^\vee\rg$ is
Cohen-Macaulay. Also an argument similar to the first paragraph of \ref{1} \give\ \ref{2}, shows that $\Delta$ has
all possible faces of size $d'$. So the assumptions of \ref{2} hold for $\Delta$ and it is easy to see that if $L$
is as in \ref{2}, then $\b{L}$ is a free MS of $\cc$. Therefore every $d$-dimensional CF-tree on $[n]$ has a free
MS and by \ref{trees have leaf}, the result follows.

\ref{2} $\iff$ \ref{3}: Suppose that $\Delta$ is a $d'$-dimensional simplicial complex on $[n]$, which is
$(d'-1)$-complete. We just need to show that $\Delta$ is Cohen-Macaulay over $\z_2$ \ifof it satisfies \ref{-2}.
Set $d_F= \dim \link_\Delta F= d'-|F|$. By the Reisner theorem (see \cite[Theorem 8.1.6]{hibi}), $\Delta$ is
Cohen-Macaulay \ifof for all $F\in \Delta$ and $i<d_F$, $\tl{H}_i(\link_\Delta F; \z_2)=0$. But by
\cite[Proposition 5.1]{Connon},  for an $x=\sum_{F\in \cc} F\in \tl{C}_i(\Delta)$ we have $x\in \ker \rnd_i$ \ifof
$\cc$ is a disjoint union of CF-cycles. Also it is easy to check that $\rnd_{i} (x)= \sum_{F\in \rnd(\cc)} F$.
Therefore, $\tl{H}_i(\link F; \z_2)=0$ \ifof each $i$-cycle of $\link F$ is the boundary of a family of
$(i+1)$-dimensional faces of $\link F$.

Each $\link F$ is $(d_F-1)$-complete (on the vertex set $[n]\sm F$), for $\Delta$ is $(d'-1)$-complete. Therefore,
$\tl{H}_i(\link F; \z_2)=0$ for $i<d_F-2$. Also if $|F|>d'-1$, then $d_F\leq 0$ and again $\tl{H}_i(\link F;
\z_2)=0$ for all $i<d_F$. So $\Delta$ is Cohen-Macaulay \ifof for each $F\in \Delta$ with $|F|\leq d'-1$, each
$(d_F-1)$-cycle of $\link F$ is the boundary of a set of facets of $\link F$. Let $A=\b{F}$. Then every
$(d_F-1)$-cycle of $\link F$ is clearly a CF-cycle on $A$. Also as $\link F$ contains all $d_F$-subsets of $A$,
the circuits of every possible $(d_F-1)$-cycle on vertex set $A$ are in $\link F$, that is, CF-cycles with
dimension $=d_F-1$ with vertex set $A$ are exactly $(d_F-1)$-cycles of $\link F$. Noting that $d_F-1=
|A|-(n-d'+1)$, the result follows.
\end{proof}
\begin{cor}\label{low n}
If $n\leq d+3$, then every $d$-dimensional CF-tree on $[n]$ has a free MS and is chordal.
\end{cor}
\begin{proof}
If $n=d+1$, then $\cc$ has just one circuit and every MS of $\cc$ is free.  Assume $d+1< n\leq d+3$. We show that
\ref{dual}\ref{2} holds. If $n=d+2$, then $d'=0$ in the previous theorem. For each $0$-dimensional simplicial
complex $\Delta$ on $[n]$ which is not $0$-complete there are vertices $i,j\in[n]$ with $\{i\}\in \Delta$ and
$\{j\}\notin \Delta$. So $L=\{i,j\}$ satisfies \ref{dual}\ref{2}. Now suppose $n=d+3$ and let $\Delta$ be a
non-1-complete 1-dimensional $\z_2$-Cohen-Macaulay simplicial complex which is 0-complete. It is known that every
Cohen-Macaulay complex with dimension $\geq 1$ is connected. As $\Delta$ is not complete, there are non-adjacent
vertices $x,y$ of $\Delta$. Suppose that $x,y$ have the least possible distance, that is, two. Then there is a
vertex $z$ adjacent to both $x$ and $y$. Now $L=\{x,y,z\}$ has the requirements of \ref{dual}\ref{2}.
\end{proof}

                          \section{Some Other Generalizations of Cycles and Chordal Graphs}

Suppose that $A$ is a set of vertices of $\cc$. We call $\{F\in \cc|F\se A\}$, the \emph{subclutter of $\cc$
induced by $A$}. Also if $M\se \MS(\cc)$, by the \emph{subclutter of $\cc$ induced by $M$} we mean $\{F\in \cc|
\MS(\{F\}) \se M\}$. For example if $\cc$ is the clutter in \oldref{fig1}, then the induced subclutter of $\cc$ on
$\MS(\cc)\sm\{ac\}$ equals $\cc\sm \{ace\}$. But if $A$ is the set of vertices of $\cc$ with $a$ and $c$ removed,
then $A$ induces the subclutter $\cc-a-c=\{bef,def\}$. Note that a subclutter induced by a set of vertices $A$, is
also induced by the set of MS's which are contained in $A$.

Let $i\in \{1,2,3\}$. We say $\cc$ is a \emph{\C{i}-cycle}, when either it is a complete clutter with $n=d+2$ or
$\sms(\cc)=\tohi$ and $\sms(\cc')\neq \tohi$ for each $\cc'\subsetneq \cc$ such that:
\begin{itemize}
\item for $i=1$, $\tohi \neq \cc'$;

\item for $i=2$, $\tohi\neq \cc'$ is a subclutter induced by a subset of $\MS(\cc)$;

\item for $i=3$, $\tohi\neq \cc'$ is a subclutter induced by a set of vertices of $\cc$.
\end{itemize}
It is easy to check that in the case $d=1$ all of these types of cycles coincide with the usual cycles of graphs.
Also clearly every \C1-cycle is a \C2-cycle and every \C2-cycle is a \C3-cycle. But the converse is not true as
the following examples show.
\begin{ex}\label{c3 not c2}
Let $\cc$ be the set of facets of $\Gamma$ in \cite[Example 16]{CF2}, that is, $\cc$ is the set of all 3-subsets
of $\{0,1,\ldots, 5\} $ except 012, 345. Then it is straightforward to check that $\sms(\cc)=\tohi$ and each
non-trivial vertex induced subclutter  of $\cc$ has a SMS. So $\cc$ is a \C3-cycle. But $\sms(\cc-12)=\tohi$, thus
$\cc$ is not a C2-cycle.
\end{ex}

To present an example of a \C2-cycle which is not a \C1-cycle, we need the following lemmas.
\begin{lem}\label{deg2 are c1}
If $\cc$ is a strongly connected $d$-clutter with $\deg(e)=2$ for all $e\in \MS(\cc)$, then $\cc$ is a \C1-cycle.
\end{lem}
\begin{proof}
Note that if $\cc$ contains a clique $\cc'$ on more than $d+1$ vertices, then $\deg_{\cc'}(e)=2$ for each $e\in
\MS(\cc')$. Therefore $\cc\sm \cc'$ does not share any MS with $\cc'$ (else, that MS has degree $>2$). But this
contradicts the strongly connectedness of $\cc$, unless $\cc\sm \cc'= \tohi$. Hence in this case $\cc$ is complete
and because each MS has degree 2, $n=d+2$. So $\cc$ is a \C1-cycle.

Thus we can assume that $\cc$ contains no cliques. So $e\in \MS(\cc)$ is a simplicial MS \ifof $\deg(e)=1$. In
particular, $\sms(\cc)= \tohi$. If $\cc'\se \cc$ has no SMS, then by an argument similar to the above paragraph,
one concludes that $\cc'=\cc$.
\end{proof}

\begin{lem}\label{|cycle|>d+1}
Suppose that $\cc$ is a $d$-dimensional CF-cycle, then $d+2\leq |\cc|$.
\end{lem}
\begin{proof}
Suppose that $F\in \cc$. Then $F$ contains $d+1$ MS's, say $e_1$, \ldots, $e_{d+1}$. As $\cc$ is a CF-cycle, $e_i$
should be contained in another circuit of $\cc$ which we call $F_i$. If for some $i\neq j$ we have $F_i=F_j$, then
$F_i$ and $F$ have two MS's in common and hence $F=F_i$, a contradiction. Thus all $F_i$'s are distinct and we
have found at least $d+2$ circuits in $\cc$.
\end{proof}

It should be mentioned that by \cite[Proposition 3.11]{CF1}, the number of vertices of a $d$-dimensional CF-cycle
is also at least $d+2$. Recall that \emph{Strong components} of $\cc$ are the maximal strongly connected
subclutters of $\cc$.
\begin{lem}\label{rnd=U cycles}
Suppose that $\cc'=\rnd(\cc)$ is non-empty. Then $\cc'$ is a disjoint union of CF-cycles.
\end{lem}
\begin{proof}
Consider $x=\sum_{F\in \cc} F\in \tl{C}_d(\lg \cc \rg, \z_2)$. Then it can easily be checked that
$y=\rnd_d(x)=\sum_{F'\in \cc'} F'$. Thus $\rnd_{d-1}(y)=0$. So according to  \cite[Propositin 5.1]{Connon}, each
strong component of $\cc'$ is a CF-cycle. Whence $\cc'$ is a disjoint union of CF-cycles.
\end{proof}

\begin{ex}\label{c2 not c1}
Let $\cc$ be the 2-clutter in \oldref{fig1}, which is a triangulation of the real projective plane. $\cc'= \cc
\cup \{abc\}$. Then $\cc'$ is a \C2-cycle but not a \C1-cycle.
\end{ex}
\begin{proof}
It follows from the previous lemma that $\cc$ is a \C1-cycle. Also it is not hard to check that no MS-induced
subclutter of $\cc$ has exactly the set $\{ab, ac, bc\}$ as the set of leaves. One can readily check that $\cc'$
has no cliques on more than 3 vertices. Thus simplicial MS's of $\cc'$ and its subclutters are exactly their
leaves. So $\sms(\cc')= \tohi$. As $\cc\se \cc'$ has no simplicial MS, we see that $\cc'$ is not a \C1-cycle.

Suppose $\cd$ is a non-trivial subclutter of $\cc'$ induced by $A\se \MS(\cc')$. If either of $ab,\, ac$ or $bc$
are in $A$, then $\cd$ is also an MS-induced subclutter of $\cc$ and has some simplicial MS. If $ab,ac,bc\notin
A$, then $\cd'= \cd\sm \{abc\}$ is an MS-induced subclutter of $\cc$ and has an MS. Note that in fact, $\cd'$ is a
CF-tree, hence $\rnd(\cd')\neq \tohi$ and \ref{rnd=U cycles} and \ref{|cycle|>d+1}, $\cd'$ has at least three MS's
with odd degree, that is, three leaves. Hence it has a free MS other than $ab$, $ac$ and $bc$, which remains free
in $\cd$, too. Consequently, every not-trivial MS-induced subclutter of $\cc'$ has a simplicial MS, whence $\cc'$
is a \C2-cycle.
\end{proof}
\begin{figure}[h]
\begin{center}
\includegraphics[scale= 0.8]{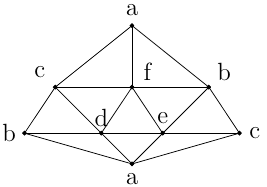}
\end{center}
\caption{$\cc$ in \ref{c2 not c1}; the circuits are the small triangles} \label{fig1}
\end{figure}

\begin{ex}
Let $\cc$ be the triangulation of the mod 3 Moore space in \oldref{fig2}. Then $\cc$ is a CF-tree without leaves.
So $\sms(\cc)=\tohi$. Also it is easy to see that every subclutter of $\cc$ has a free MS. Hence $\cc$ is a
\C1-cycle (hence \C2 and \C3) which is not a CF-cycle. If $\cc'=\cc\cup \{xyz\}$, then $\cc'$ is a CF-cycle which
is not a \C1-cycle, for the subclutter $\cc$  of $\cc'$ has no SMS's.
\end{ex}

We saw that there are clutters without any CF-cycle which do not have SMS's. But if we replace CF-cycle with \C
i-cycle, then we have:
\begin{prop}\label{acyclic => leaf}
Assume that $i\in [3]$ and no subclutter of $\cc$ is a \C{i}-cycle, then $\cc$ has a free MS. In particular, $\cc$
is chordal and $I(\b\cc)$ has a linear resolution.
\end{prop}
\begin{proof}
It suffices to prove the claim for $i=1$. Suppose that $\sms(\cc')= \tohi$ and $\cc'$ is a minimal non-empty
subclutter of $\cc$ with this property. Then $\cc'$ is a \C1-cycle, a contradiction. Hence $\sms(\cc')\neq \tohi$
for each non-empty subclutter $\cc'$ of $\cc$. Also if $\cc$ has a clique on more than $\dim \cc+1$ vertices, then
$\cc$ contains a complete \C1-cycle. Thus $\cc$ has no such cliques and whence if $e\in \sms(\cc)$, then $e$
should be free. So every non-empty subclutter of $\cc$, including  $\cc$ itself, has a free MS.
\end{proof}

A graph is chordal \ifof no induced subgraph is a non-complete cycle. Thus using either of the above notions of
cycle and using either inducing by vertices or inducing by $\MS$'s, we get generalizations of chordal graphs to
clutters. But the following examples show that none of these generalizations preserve the Fr\"oberg theorem.

\begin{ex}\label{not chordal by ci}
Let $\cc$ be the octahedron, that is, $\cc= \{1e,2e|e=34,45,56,63\}$.
\begin{enumerate}
\item $\cc'=\cc \cup \{135, 235, 435, 635\}$ is chordal and hence $I(\b{ \cc'})$ has a linear resolution over
    every field. But $\cc= \cc'-\{35\}$ is an MS-induced non-complete \C{i}-cycle ($i\in [3]$) and also CF-cycle
    of $\cc'$.

\item If $\cc''= \cc \cup \{435\}$, then every non-empty vertex induced subclutter of $\cc''$ has a free MS. So
    $\cc''$ has no vertex induced \C{i}-cycles ($i\in [3]$) or CF-cycles. But $\cc$ is not chordal and $I(\b{
    \cc''})$ has not a linear resolution over $\z_2$ because it is not CF-chorded.
\end{enumerate}
\end{ex}
On the positive side we have:
\begin{prop}\label{ci and chordality}
Consider the following statements on a $d$-clutter $\cc$.
\begin{enumerate}
\item No MS-induced subclutter of $\cc$ is a non-complete C$_2$-cycle.

\item $\cc$ is chordal.

\item No vertex induced subclutter of $\cc$ is a non-complete C$_3$-cycle.
\end{enumerate}
Then \ref{1} \give\ \ref{2} \give\ \ref{3}.
\end{prop}
\begin{proof}
\ref{1} \give\ \ref{2}: It is easy to see that \ref{1} is equivalent to $\sms(\cc') \neq \tohi$ for each non-empty
MS-induced subclutter $\cc'$ of $\cc$. So there is $e_1\in \sms(\cc)$, $e_2\in \sms(\cc-e_1)$, \ldots and $\cc$ is
chordal.

\ref{2} \give\ \ref{3}: Again \ref{3} is equivalent to $\sms(\cc') \neq \tohi$ for each non-empty vertex induced
subclutter $\cc'$ of $\cc$. As any vertex induced subclutter is obtained by consecutively deleting some vertices,
we just need to show that if $\cc-v\neq \tohi$, then $\sms(\cc-v)\neq \tohi$ for a chordal clutter $\cc$ and
vertex $v$ of $\cc$. We prove this by induction on $|\cc|$.

The case $|\cc|=1$ is trivial. Suppose $|\cc|>1$. Let $e\in \sms(\cc)$. If $v\notin e$, then $e\in \sms(\cc-v)$.
If $v\in e$, then $\cc-v= \cc -e -v$ and $\cc-e$ is a chordal clutter with a smaller number of circuits. Hence the
result follows from the induction hypothesis.
\end{proof}

It should be noted that if in \ref{ci and chordality}\ref{3}, we replace \C3 with \C2 or \C1, again the result
clearly holds. Also if we replace \C3-cycle with face-minimal CF-cycle, then again the result holds, since chordal
clutters are CF-chorded and it immediately follows the definition of CF-chorded clutters that such clutters can
not have induced non-complete face-minimal CF-cycles. But if we replace \C2-cycle in \ref{1} with \C1-cycle or
CF-cycle, then the obtained statement is not generally true. For example, $\cc'$ of \ref{c2 not c1} contains no
MS-induced \C1-cycle or CF-cycle but it is not chordal.

                                             
\includepdf[pages={1,2}]{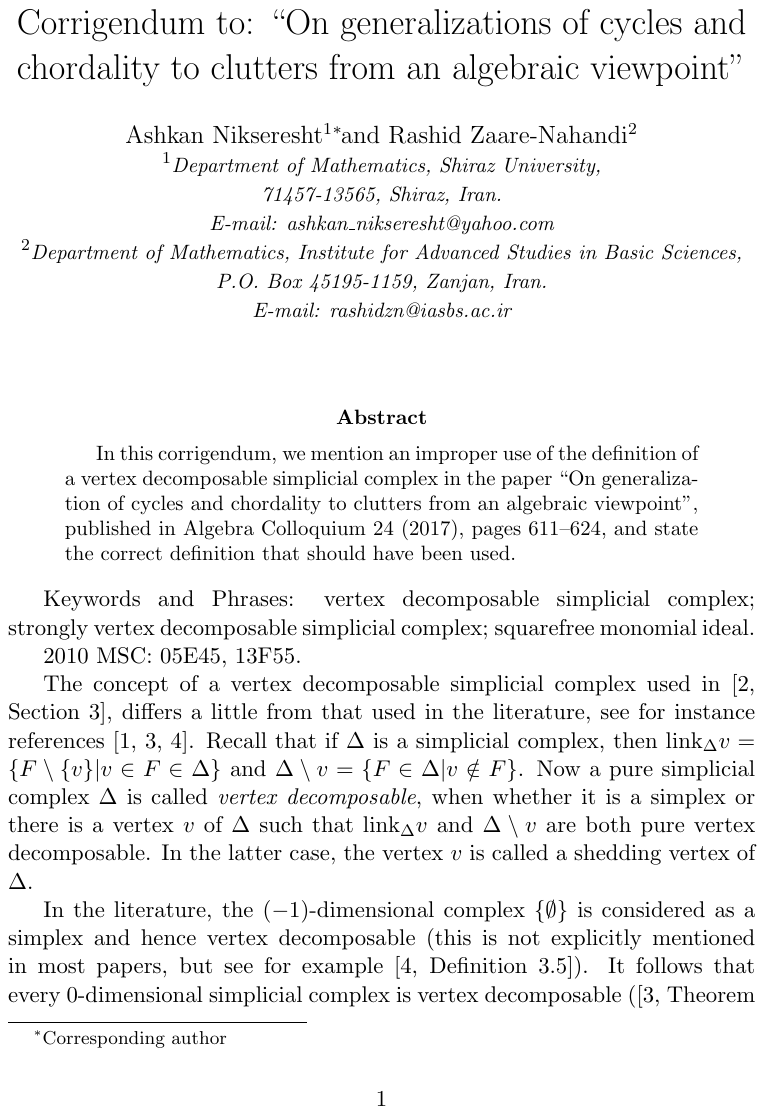}
\end{document}